\documentclass[11pt,onecolumn]{article}
\usepackage{graphicx}
\usepackage{cite} % for pdf, bitmapped graphics files
\usepackage{amsmath, amssymb, amsthm, fullpage}
\usepackage{epstopdf}
\newcommand{\Real}{\mathbb R}
\newcommand{\Natural}{\mathbb N}
\newcommand{\etal}{et. al.\,}

\usepackage{mathtools}
\usepackage{cuted}

\newtheorem{theorem}{Theorem}
\newtheorem{proposition}{Proposition}
\newtheorem{corollary}{Corollary}

\usepackage{standalone, amsmath, amssymb, graphicx, color}
\usepackage{url}
\newcommand{\vo}[1]{\boldsymbol{#1}} % Vector object
\newcommand{\mo}[1]{\boldsymbol{#1}} % Matrix object

\newcommand{\real}{\mathbb{R}}

 \newtheorem{lemma}[theorem]{Lemma}
\newcommand{\Exp}[1]{\boldsymbol{\mathsf{E}} \left[#1\right]}
\newcommand{\E}[1]{\boldsymbol{\mathsf{E}} \left[#1\right]}

\newcommand{\uniform}[1]{\mathcal{U}_{#1}}

\newcommand{\x}{\vo{x}} % state
\newcommand{\xdot}{\dot{\vo{x}}} % state derivative
\renewcommand{\u}{\vo{u}} % output
\newcommand{\param}{\vo{\Delta}} %parameters
\newcommand{\domain}[1]{{\mathcal{D}_{#1}}} %parameters
\newcommand{\Y}{\vo{Y}} % Multiple observations
\newcommand{\A}{\mo{A}} % A matrix of linear system
\let\B=\undefined
\newcommand{\B}{\mo{B}} % B matrix of linear system

\usepackage{xifthen}% provides \isempty test
\newcommand{\basis}[2]{%
 \phi_{#1}
  \ifthenelse{\isempty{#2}}%
    {}% if #1 is empty
    {({#2})}% if #1 is not empty
}

\newcommand{\xpc}{\x_{pc}} % state
\newcommand{\xpcdot}{\dot{\x}_{pc}} % state
\newcommand{\Apc}{\mo{A}_{pc}} % A matrix
\newcommand{\I}[1]{\vo{I}_{{#1}}} % identity

\renewcommand{\vec}[1]{\boldsymbol{\mathsf{vec}}\left({#1}\right)}

\newcommand{\X}{\mo{X}} % state
 
\newcommand{\W}{\mo{W}} 
\newcommand{\K}{\mo{K}} % Feedback gain.

\newcommand{\set}[1]{\mathcal{#1}}
\newcommand{\inner}[1]{\left\langle #1 \right\rangle}

\newcommand{\figlabel}[1]{\label{fig:#1}}
\newcommand{\eqnlabel}[1]{\label{eqn:#1}}

\newcommand{\eqn}[1]{\eqref{eqn:#1}}

\newcommand{\fig}[1]{fig.(\ref{fig:#1})}
\newcommand{\Fig}[1]{Fig.(\ref{fig:#1})}

\definecolor{darkgreen}{rgb}{0,0.65,0}

\renewcommand{\P}{\mo{P}} % state
\newcommand{\Kp}{\K(\param)} % param dep K

\newcommand{\Q}{\mo{Q}} % state Cost
\newcommand{\R}{\mo{R}} % Control Cost

\newcommand{\Acl}{\A_\textit{cl}}

\newcommand{\Phin}[1]{\mo{\Phi}_{#1}}

\newcommand{\sym}[1]{\mathsf{sym}\left(#1\right)}
\newcommand{\zpc}{\vo{z}_{pc}}
\newcommand{\zsc}{\vo{z}_{sc}}

\newcommand{\tr}{\mathbf{tr\;}}
\newcommand{\M}{\mo{M}}
\newcommand{\Pp}{\P(\param)}
\newcommand{\Yp}{\Y(\param)}
\newcommand{\Wp}{\W(\param)}
\newcommand{\Yb}{\bar{\Y}}

\newcommand{\mY}{\mathcal{Y}}
\newcommand{\mW}{\mathcal{W}}

\newcommand{\VYb}{\mo{V}_{\Yb}}
\newcommand{\VW}{\mo{V}_{\W}}

\newcommand{\mVW}{\mathcal{V}_{\W}}

\newcommand{\mVYb}{\mathcal{V}_{\Yb}}

\newcommand{\Ln}[1]{\vo{L}_{#1}}
\newcommand{\Yt}{\tilde{\Y}}

\newcommand{\ba}{\begin{eqnarray}}
\newcommand{\ea}{\end{eqnarray}}
\begin{document}

\title{Design of Linear Parameter Varying Quadratic Regulator in Polynomial Chaos Framework
}
\author{\begin{tabular}{cc}Shao-Chen Hsu & Raktim Bhattacharya\end{tabular} \\ \small Laboratory for Uncertainty Quantification, \\ \small Department of Aerospace Engineering, Texas A\&M University.
\thanks{The work of Shao-Chen Hsu was supported by the TIAS Award of Heep Fellowship.}
\thanks{Shao-Chen Hsu and Raktim Bhattacharya are with the Department of
Aerospace Engineering, Texas A$\&$M University, College Station, TX 77843-
3141, USA, $\lbrace$addyhsu, raktim$\rbrace$@tamu.edu. }
}

\maketitle

\begin{abstract}
We present a new theoretical framework for designing linear parameter varying controllers in the polynomial chaos framework. We assume the scheduling variable to be random and apply polynomial chaos approach to synthesize the controller for the resulting linear stochastic dynamical system. Two algorithms are presented that minimize the performance objective with respect to the stochastic system. The first algorithm is based on Galerkin projection and the second algorithm is based on stochastic collocation. LPV controllers from both the algorithms are shown to outperform classical LPV designs with respect to regulator design for nonlinear missile system.
\end{abstract}

\section{Introduction}

Linear parameter varying (LPV) systems are of the form 
\begin{equation}
\xdot = \A(\rho)\x + \B(\rho)\u,\eqnlabel{lpv}
\end{equation}

where system matrices depend on unknown parameter $\rho(t)$, which is measurable in real-time \cite{shamma2012overview, leith2000survey}. Many nonlinear systems can be transformed to LPV systems and control systems can be designed using parameter dependent convex optimization problems. Typically, parameter dependent quantities are approximated using a known class of functions such as multilinear basis functions of $\rho$, linear fractional transformations of system matrices, or by gridding the parameter space. Both these approaches result in solution of a finite, but possible large, number of linear matrix inequalities (LMIs). Further, the choice of the basis functions or the resolution of the grid could lead to conservatisms in the design. Clearly, there is a tradeoff between problem size and conservatism in the design \cite{toker1997complexity}. 

Fujisaki \etal \cite{fujisaki2003probabilistic} addressed the computational complexity of such problems by presenting a probabilistic approach to solve these problems, via a sequential randomized algorithm, which significantly reduces the computational complexity. Here the parameter $\rho(t)$ is assumed to be bounded i.e. $\rho(t) \in \domain{\rho} \subset \real^d$ and is treated as a random variable, with a distribution $f_{\rho}(\rho)$ defined over $\domain{\rho}$. The LPV synthesis problem is solved by sampling $\domain{\rho}$ and solving the sampled LMIs using a sequential-gradient method. As with any probabilistic algorithm, there is a tradeoff between sample complexity and confidence in the solution. Often, a large number of samples are required to generate a solution with high confidence.  Also, the LMIs depend only on $\rho(t)$ and not in $\dot{\rho}(t)$ as it is in classical LPV formulation.

This paper is motivated by the work of Fujisaki \etal $\;$ and is based on the idea of treating $\rho$ as a random variable. Therefore, by substituting $\rho\equiv \param$ in the system equation, we get 
\begin{align}
 \dot{\x}	 &=  \A(\param)\x + \B(\param)\u, \eqnlabel{contiDyn}
\end{align}
where $\param\in\real^d$ is a vector of uncertain parameters, with joint probability density function $f_{\param}(\param)$. Matrices $\A(\param)\in\real^{n\times n}$, $\B(\param) \in\real^{n\times m}$ are system matrices that depend on $\param$. Consequently, the solution $\x:=\x(t,\param)\in\real^n$ also depends on $\param$. Like in \cite{fujisaki2003probabilistic} we ignore temporal variation in the parameter and thus treat $\param$ as random variables. Thus, we now study the system in \eqn{lpv} as a \textit{linear time invariant system} with probabilistic system parameters. The LPV control design objective is then equivalent to designing a state-feedback law of the form $\u = \K(\param)\x$, which stabilizes the system in some suitable sense, where $\K(\param)\in\real^{m\times n}$. Thus, we are looking to obtain a parameter dependent gain $\K(\param)$ that stabilizes the system in \eqn{contiDyn} and optimizes a certain performance index. The closed-loop system is then
\begin{align}
\dot{\x}	 &=  \left[\A(\param) + \B(\param)\K(\param)\right]\x, \nonumber \\
			 &=  \Acl(\param)\x. \eqnlabel{Acl}
\end{align}

There are two distinct differences between the work presented here and that in \cite{fujisaki2003probabilistic}. We do not use a randomized approach to solve the stochastic problem, and thus don't have issues related to confidence in the solution. In our approach, the stochastic problem is solved using polynomial chaos theory, which is a deterministic approach as described later. In addition, stability of the LPV system is formulated 
in an optimal way which minimizes a cost-to-go function for the corresponding stochastic system.
%as an exponential mean square stability problem for the corresponding stochastic system. 
In \cite{fujisaki2003probabilistic}, stability of the LPV system is formulated in the probabilistic sense. 

Main contribution of this paper is an LPV regulator synthesis algorithm, in the polynomial chaos framework, which generates a parameter dependent gain. The controller is optimal with respect to a quadratic cost in state and control. The paper is organized as follows. We first provide a brief background on polynomial chaos theory and show how it is applied to study linear dynamical systems with random parameters. 
This is followed by conditions for optimal regulation in the polynomial chaos framework for closed-loop system with parameter dependent controller.
%This is followed by conditions for exponential mean-square stability in the polynomial chaos framework for closed-loop systems with parameter dependent controller. 
This leads to the main controller synthesis related results in the paper. The paper ends with an example nonlinear missile system that highlights the superiority of the polynomial chaos approach over the classical LPV design approach.

\section{Polynomial Chaos Theory}
Polynomial chaos (PC) is a \textit{deterministic} method for evolution of uncertainty in dynamical system, when there is probabilistic uncertainty in the system parameters. Polynomial chaos was first introduced by Wiener \cite{wiener}
where Hermite polynomials were used to model stochastic processes
with Gaussian random variables. It can be thought of as an extension of Volterra's theory of nonlinear functionals for stochastic systems \cite{volterra,pcFEM}. According to Cameron and Martin \cite{CameronMartin} such an expansion converges in the $\mathcal{L}_2$ sense for any arbitrary stochastic process with
finite second moment. This applies to most physical systems. Xiu
\etal \cite{xiu:02} generalized the result of Cameron-Martin to various
continuous and discrete distributions using orthogonal polynomials
from the so called Askey-scheme \cite{Askey-Polynomials} and
demonstrated $\mathcal{L}_2$ convergence in the corresponding Hilbert
functional space. The PC framework has been applied to applications including
stochastic fluid dynamics \cite{pcFluids2,pcFluids4,pcFluids5},
stochastic finite elements \cite{pcFEM}, and solid mechanics
\cite{pcSolids1,pcSolids2}, feedback control \cite{hover2006application, kim2012generalized, fisher2009linear, bhattacharya2012linear} and estimation \cite{Dutta2010}. It has been shown that PC based methods are computationally far superior than Monte-Carlo based methods \cite{xiu:02, pcFluids2, pcFluids4, pcFluids5, le2010spectral}. See \cite{eldred2009comparison} for several benchmark problems.

Formally, the PC framework is described as follows. Let $(\Omega,\mathcal{F},P)$ be a probability space, where $\Omega$
is the sample space, $\mathcal{F}$ is the $\sigma$-algebra of the
subsets of $\Omega$, and $P$ is the probability measure. Let
$\param(\omega) =
(\param_1(\omega),\cdots,\param_d(\omega)):(\Omega,\mathcal{F})\rightarrow(\Real^d,\mathcal{B}^d)$
be an $\Real^d$-valued continuous random variable, where
$d\in\Natural$, and $\mathcal{B}^d$ is the $\sigma$-algebra of Borel
subsets of $\Real^d$. 

A general second order process $X(\omega)\in
\mathcal{L}_2(\Omega,\mathcal{F},P)$ can be expressed by polynomial
chaos as
\begin{equation}
\eqnlabel{gPC}
X(\omega) = \sum_{i=0}^{\infty} x_i\phi_i({\param}(\omega)),
\end{equation}
where $\omega$ is the random event and $\phi_i({\param}(\omega))$
denotes the polynomial chaos basis of degree $p$ in terms of the random variables
$\param(\omega)$. In practice, the infinite series is truncated and $X(\omega)$ is approximated by 
\[
X(\omega) \approx \hat{X}(\omega) = \sum_{i=0}^{N} x_i\phi_i({\param}(\omega)).
\] The functions $\{\phi_i\}$ are a family of
orthogonal basis in $\mathcal{L}_2(\Omega,\mathcal{F},P)$ satisfying
the relation
\begin{align}
\Exp{\phi_i\phi_j}:= \int_{\mathcal{D}_{\param}}\hspace{-0.1in}{\basis{i}{\param}\basis{j}{\param} f_{\param}(\param)
\,d\param}   = \begin{cases}
    0,& \text{if } i \neq j,\\
    \int_{\domain{\param}} \phi_i^2f_{\param}(\param)\,d\param, & \text{otherwise.}
\end{cases}\eqnlabel{basisFcn}
\end{align}
where $\mathcal{D}_{\param}$ is the domain of the random variable $\param(\omega)$, and
$f_{\param}(\param)$ is a probability density function for $\param$. Table \ref{table.pc} shows the family of basis functions for random variables with common distributions.
\begin{table}[htbp]
\begin{center}
\begin{tabular}{|c|c|}
\hline
Random Variable $\param$ & $\phi_i(\param)$ of the Wiener-Askey Scheme\\ \hline
Gaussian & Hermite \\
Uniform  & Legendre \\
Gamma   & Laguerre \\
Beta    & Jacobi\\
\hline
\end{tabular}
\end{center}
\caption{Correspondence between choice of polynomials and given
distribution of $\param(\omega)$ \cite{xiu:02}.} 
\label{table.pc}
\end{table}

Generally, there are two methods for expanding a random process in this framework -- Galerkin projection and stochastic collocation. These two approaches are described next.

%\subsection{Application to Dynamical Systems with Random Parameters}
\subsection{Galerkin Projection}
With respect to the dynamical system defined in \eqn{contiDyn}, the solution can be approximated by the polynomial chaos expansion as
\begin{align}
	\x(t,\param) \approx \hat{\x}(t,\param) =  \sum_{i=0}^N \x_i(t)\basis{i}{\param},
\end{align}
where the polynomial chaos coefficients $\x_i \in \real^n$. Define $\mo{\Phi}(\param)$ to be
\begin{align}
\mo{\Phi} &\equiv \mo{\Phi}(\param) := \begin{pmatrix}\basis{0}{\param} & \cdots & \basis{N}{\param}\end{pmatrix}^T, \text{ and } \\
\mo{\Phi}_n &\equiv \mo{\Phi}_n(\param) := \mo{\Phi}(\param) \otimes \I{n},
\end{align}
where $\I{n}\in\real^{n\times n}$ is identity matrix. Also define matrix $\X\in\real^{n\times(N+1)}$, with polynomial chaos coefficients $\x_i$, as
\[ \X = \begin{bmatrix} \x_0 & \cdots & \x_N \end{bmatrix}.\]

This lets us define $\hat{\x}(t,\param)$ as 
\begin{align}
\hat{\x}(t,\param) := \X(t)\mo{\Phi}(\param) \eqnlabel{compactX}.
\end{align}
Noting that $\hat{\x} \equiv \vec{\hat{\x}}$, we obtain an alternate form for \eqn{compactX},
\begin{align}
\hat{\x} \equiv  \vec{\hat{\x}}  = \vec{\X\mo{\Phi}}  = \vec{\I{n}\X\mo{\Phi}}  = (\mo{\Phi}^T\otimes \I{n})\vec{\X} = \mo{\Phi}_n^T\xpc, \eqnlabel{compactxpc}
\end{align}
where  $\xpc := \vec{\X}$, and $\vec{\cdot}$ is the vectorization operator \cite{horn2012matrix}.

Since $\hat{\x}$ from \eqn{compactxpc} is an approximation, substituting it in \eqn{Acl} we get equation error $\vo{e}$, which is given by 
\begin{align}
\vo{e} &:= \dot{\hat{\x}} - \Acl(\param)\hat{\x}  =  \mo{\Phi}_n^T\xpcdot - \Acl(\param)\mo{\Phi}_n^T\xpc.
\end{align}
Best $\set{L}_2$ approximation is obtained by setting
\begin{align}
\inner{\vo{e}\phi_i} := \Exp{\vo{e}\phi_i}=0, \text{ for } i = 0,1,\cdots,N.
\end{align}
\begin{align}
&\Exp{\Phin{n}\Phin{n}^T}\xpcdot = \Exp{\Phin{n}\Acl\Phin{n}^T}\xpc,\nonumber \\
& \implies \xpcdot  = \Exp{\Phin{n}\Phin{n}^T}^{-1}\Exp{\Phin{n}\Acl\Phin{n}^T}\xpc,  \eqnlabel{pcDynamics}\\
&\text{or } \xpcdot  = \Apc \xpc.\eqnlabel{pcAcl}
\end{align}
where $\Phin{n}$ and $\Acl$ depend on $\param$ as defined earlier. Equation \eqn{pcAcl}, is best finite dimensional approximation of \eqn{Acl} in the $\set{L}_2$ sense.

\subsection{Stochastic Collocation}
In this approach, we introduce Lagrangian interpolants $$l_i(\param)=\prod^N_{j=0,j\neq i}\frac{\param-\param_j}{\param_i-\param_j}$$ as basis functions, where $\param_i$ are the roots of the polynomial chaos basis of degree $N+1$. The Lagrange interpolants are orthogonal to each other in the $\mathcal{L}_2$ sense, which can be proved using  Gaussian quadrature rule as follows.
\begin{align*}
\E{l_i l_j}&=\int_{\mathcal{D}_{\param}}l_i(\param)l_j(\param)f_{\param}(\param)d\param\\
&=\sum_{k=0}^{N}\int_{\mathcal{D}_{\param}}l_k(\param)f_{\param}(\param)d\param l_i(\param_k)l_j(\param_k).
\end{align*}
Since $i\neq j$, we can conclude
\begin{align*}
\E{l_i l_j}=0.
\end{align*}

The solution to the dynamical system \eqn{Acl} thus can also be approximated as 
\begin{align*}
\x(t,\param) \approx \tilde{\x} = \sum_{i=0}^N \x_{i,sc}(t)l_{i}(\param)=\Ln{n}^T\x_{sc},
\end{align*}
where $\Ln{n}=\left[l_0\ldots l_N\right]^T\otimes\I{n}$, and $\x_{sc}=[\x_{0,sc} \ldots \x_{N,sc} ]^T$ are coefficients determined by solving $\dot{\x}_{i,sc}=\A_{cl}(\param_i)\x_{i,sc}$. It implies that the solution $\tilde{\x}$ is exact at those specified sample points, which means that the error $\tilde{\vo{e}}=\dot{\tilde{\x}}-\A_{cl}(\param)\tilde{\x}$ is forced to be zero at the sample points $\param_i$. 

We need the following result to derive the optimal control law  in the stochastic collocation setting.

\begin{lemma}
Consider two Lagrangian interpolants $l_i(\param)$ and $l_j(\param)$ , and a function $g(\param)$, then
\begin{align}
\E{l_i(\param)l_j(\param)g(\param)} =
\begin{cases}
0, &i\neq j\\
\E{l_i(\param)}g(\param_i), &i=j
\end{cases},
\label{eq:gaussian}
\end{align}
if $g(\param)$ is affine in $\param$.
\label{lem:gaussian}
\end{lemma}
\begin{proof}
\begin{align*}
&\E{l_i(\param)l_j(\param)g(\param)}=\int_{\mathcal{D}_{\param}}l_i(\param) l_j(\param) g(\param)f_{\param}(\param)d\param\\
&\approx\sum_{m=0}^{N}\int_{\mathcal{D}_{\param}}l_m (\param)f_{\param}(\param)d\param l_i(\param_m) l_j(\param_m) g(\param_m)\\ %&&(\text{Gaussian quadrature})\\
&=\begin{cases}
0& i\neq j\\
\int_{\mathcal{D}_{\param}}l_i(\param) f_{\param}(\param)d\param g(\param_i)&i=j
\end{cases}\\
&=\begin{cases}
0&i\neq j\\
\E{l_i(\param)}g(\param_i)&i=j
\end{cases},
\end{align*}
According to Gaussian quadrature rule, the expression is exact when $l_i l_j g$ is a polynomial of degree  at most $2N+1$. It is known that the Lagrangian interpolants are N-th order polynomials, so we can conclude that if $g$ is affine in $\param$, \eqref{eq:gaussian} is exact.
\end{proof}

We next present synthesis of optimal control law using both Galerkin projection and stochastic collocation approach.

\section{Optimal Controller Synthesis}
Given the system in \eqn{contiDyn} we are interested in state feedback control $\u:=\K(\param)\x(t,\param)$ that minimizes the following cost function
\begin{align}
\min_{\u} \E{\int_0^\infty (\x^T \Q \x + \u^T\R\u)\,dt}.
\eqnlabel{costFcn}
\end{align}

If $\exists$ $V(\x)>0$ such that 
\begin{align}
\E{\frac{dV}{dt}} \leq - \E{\x^T \Q \x + \u^T\R\u}. \eqnlabel{condition}
\end{align}
Integrating from $[0,T]$ gives us
\begin{align*}
\int_0^T \E{\frac{dV}{dt}}dt &\leq -\int_0^T \E{\x^T \Q \x + \u^T\R\u}dt, \\
\E{V(\x(T))} - \E{V(\x(0))} &\leq -\int_0^T \E{\x^T \Q \x + \u^T\R\u}dt, 
\end{align*}
since $\E{V(\x(T))} \geq 0$ implies 
\begin{align}
- \E{V(\x(0))} \leq -\int_0^T \E{\x^T \Q \x + \u^T\R\u}dt, \text{ for all $T>0$},
\end{align}
or,
\begin{align}
\E{V(\x(0))} \geq \int_0^\infty \E{\x^T \Q \x + \u^T\R\u}dt. 
\end{align}
Therefore, \eqn{condition} provides a sufficient condition for upper bound on the cost-to-go.

In this paper, we use polynomial chaos theory to determine the expectation operator in \eqn{condition}. We apply both Galerkin projection and stochastic collocation techniques, and derive control synthesis problem in the respective frameworks, for the system considered in \eqn{contiDyn}. We will see later, the Galerkin projection is more accurate than stochastic collocation technique, but results in more complex synthesis problems. Consequently, the computational time for synthesis is more, but generates better performing controller.

Before we proceed, we  need the following result in the rest of the paper.
\begin{proposition} For any vector $\vo{v} \in\real^{N+1}$ and matrix $\vo{M}\in\real^{m\times n}$
\begin{equation}
		\mo{M}(\mo{v}^T\otimes \I{n}) = (\mo{v}^T \otimes \I{m}) (\I{N+1}\otimes\mo{M}), \eqnlabel{prop3}
	\end{equation}
	where $\I{\ast}$ is identity matrix with indicated dimension.
\end{proposition}
\begin{proof}
	\begin{align*}
		\mo{M}(\mo{v}^T \otimes \I{n})  & = (1 \otimes \mo{M})(\mo{v}^T \otimes \I{n})  \\ &= \mo{v}^T \otimes \mo{M} = (\mo{v}^T\I{N+1}) \otimes (\I{m}\mo{M})\\
		&= (\mo{v}^T \otimes \I{m})(\I{N+1}\otimes \mo{M}).	\end{align*}
\end{proof}

\subsection{Galerkin Projection Based Formulation}
Here Galerkin projection is used to solve the stochastic optimal control problem. Using the parameterization given by \eqn{compactxpc}, and the sufficient condition in \eqn{condition}, we present the following optimal control law synthesis algorithm.

\begin{theorem} \label{thmMain}
Controller gain $\Kp:= \Wp\Y^{-1}(\param)$ minimizes \eqn{costFcn} if $\exists$ $\Yp=\Y^T(\param)>0\in\real^{n\times n}$ and $\Wp=\W^T(\param)>0\in\real^{m\times n}$,  which are the solutions of the optimization problem
\begin{align}
\max \tr \E{\Yp} \eqnlabel{cost}
\end{align}
subject to
\begin{align}
\E{\sym{\mY\Phin{n}\A^T\Phin{n}^T +  \mW^T\Phin{m}\B^T\Phin{n}}+ \mY\Phin{n}\Q\Phin{n}^T\mY} +  \E{\mW^T\Phin{m}\R\Phin{m}^T\mW} \leq 0, \eqnlabel{constr}
\end{align}
where $\mY:=\I{N+1}\otimes \Yp$ and $\mW:=\I{N+1}\otimes \Wp$, and $\sym{\vo{X}}:=\vo{X}+\vo{X}^T$.
\end{theorem}

\begin{proof}
Let $V(\x):=\x^T\Pp\x, \Pp = \P^T(\param) > 0 \in \real^{n\times n}$. Therefore, \eqn{condition} can be written as
\begin{align*}
\E{\x^T\left(\sym{(\A + \B\K)^T\P} + \Q + \K^T\R\K\right)\x} \leq 0,
\end{align*}
or
\begin{align}
\E{\x^T\left( \sym{\A^T\P + \K^T\B^T\P} + \Q + \K^T\R\K\right)\x} \leq 0.
\eqnlabel{quad1}
\end{align}

Substituting $\P := \Y^{-1}$ and $\x := \Y \vo{z}$, in the above quadratic form, we get 
\begin{align*}
&\x^T\left(\sym{\P\A + \P\B\K}  + \Q + \K^T\R\K\right)\x\\
& = \vo{z}^T \Y \left(\sym{\Y^{-1}\A + \Y^{-1}\B\K} + \Q + \K^T\R\K\right)\Y \vo{z}\\
& = \vo{z}^T \left(\sym{\A\Y + \B\K\Y} + \Y\Q\Y + \Y\K^T\R\K\Y\right) \vo{z}
\end{align*}
With $\W := \K\Y$, the quadratic form simplifies to
$$
\vo{z}^T \left(\sym{\A\Y +  \B\W} + \Y\Q\Y + \W^T\R\W\right) \vo{z}.
$$
Therefore, the condition in \eqn{quad1} is equivalent to 
\begin{align}
\E{\vo{z}^T \left(\mathsf{sym}(\A\Y + \B\W) + \Y\Q\Y + \W^T\R\W\right) \vo{z}} \leq 0.
\eqnlabel{quad2}
\end{align}

Assuming, $\vo{z}(t,\param)$ is a second order process, we can represent $$\vo{z}(t,\param) :=  \sum_{i=0}^\infty \x_i(t)\basis{i}{\param} = \Phin{n}^T(\param)\zpc(t).$$
Note that we have included infinite terms in the polynomial expansion, and thus the representation is exact. %The theory is presented with the exact, infinite term, polynomial chaos expansion. The finite term truncation is needed only for computation.\\

With $\vo{z} = \Phin{n}^T\zpc$, we get 
\begin{align*}
\zpc^T  \boldsymbol{\mathsf{E}}\left[\sym{\Phin{n}\Y\A^T\Phin{n}^T + \Phin{n}\W^T\B^T\Phin{n}^T} + \Phin{n}\Y\Q\Y\Phin{n}^T +  \Phin{n}\W^T\R\W\Phin{n}^T\right]\zpc\leq 0,
\end{align*}
or 
\begin{align*}
\boldsymbol{\mathsf{E}}\left[\sym{\Phin{n}\Y\A^T\Phin{n}^T + \Phin{n}\W^T\B^T\Phin{n}^T} + \Phin{n}\Y\Q\Y\Phin{n}^T +  \Phin{n}\W^T\R\W\Phin{n}^T\right]\leq 0.
\end{align*}

Using \eqn{prop3} we can write $\Phin{n}\Y = \mY\Phin{n}$ and $\Phin{n}\W^T = \mW^T\Phin{m}$. Substituting them, we get 
\begin{align*}
\boldsymbol{\mathsf{E}}\left[\sym{\mY\Phin{n}\A^T\Phin{n}^T +  \mW^T\Phin{m}\B^T\Phin{n}}  + \mY\Phin{n}\Q\Phin{n}^T\mY  +  \mW^T\Phin{m}\R\Phin{m}^T\mW \right] \leq 0.\end{align*}

Since $\x_0$ is given, with no initial condition uncertainty, the cost function can be written as
$$ \E{\x_0^T\Pp\x_0} = \x_0^T\E{\Pp}\x_0 = \x_0^T\E{\Y^{-1}(\param)}\x_0.$$
Therefore for a given $\x_0$,
\begin{align*}  
\max \tr \E{\Yp} &\implies \min \tr \E{\Pp}\\ & \implies  \min \x_0^T\E{\Pp}\x_0
\end{align*}
\end{proof}

The matrix variables $\Yp$ and $\Wp$ in \eqn{constr} are infinite dimensional.  In this paper, we consider finite dimensional parameterization of $\Yp=\Yp^T>0$ from the literature on sum-of-square (SOS) representation of matrix polynomials \cite{scherer2006matrix}, defined by the following.

\begin{lemma}(Lemma 1 in \cite{scherer2006matrix})\label{lemYb}
The polynomial matrix $\Yp$ of dimension $n\times n$ is SOS with respect to the monomial basis $\mo{\Phi}(\param)$ iff there exists a symmetric matrix $\Yb$ such that
\begin{align}
\Yp = \left(\mo{\Phi}(\param) \otimes \I{n} \right)^T \Yb \left(\mo{\Phi}(\param) \otimes \I{n}\right)  = \Phin{n}^T(\param)\Yb\Phin{n}(\param), \text{ and } \Yb \in \real^{n(N+1)\times n(N+1)} \geq 0.
\eqnlabel{lemYb}
\end{align}
\end{lemma}

\begin{proof}
See Lemma 1 in \cite{scherer2006matrix}.
\end{proof}
Therefore, $\Yp = \Yp^T>0, \;\forall \param \;\Leftrightarrow\; \Yb = \Yb^T > 0$. 

\begin{corollary}
$\Yb$ in \eqn{lemYb} admits a linear space of $n(n+1)(N+1)(N+2)/4$.
\end{corollary}

\begin{proof}
Partition $\Yb$ as 
\begin{align}
\Yb := \begin{bmatrix}\Yb_{00} &\cdots & \Yb_{0N} \\
\vdots &  & \vdots \\
\Yb_{N0} &\cdots &\Yb_{NN}\\
\end{bmatrix}, \eqnlabel{Ybar}
\end{align}
where $\Yb_{ij} = \Yb_{ji}^T \in \real^{n\times n}$. Therefore,
\begin{align}
\Yp &:=\Phin{n}^T\Yb\Phin{n} = \sum_{ij}\phi_i(\param)\phi_j(\param)\Yb_{ij}. 
\eqnlabel{simpleYb}
\end{align}
%\nonumber \\ 
%& = \begin{bmatrix} \phi_0(\param)\I{n} & \cdots &  \phi_N(\param)\I{n}\end{bmatrix}
%\begin{bmatrix}\Yb_{00} &\cdots & \Yb_{0N}\\
%\vdots &  & \vdots \\
%\Yb_{N0} &\cdots &\Yb_{NN}\\
%\end{bmatrix}
%\begin{bmatrix} \phi_0(\param)\I{n} &\\ \vdots \\ \phi_N(\param)\I{n}\end{bmatrix} \\ 
%& = \sum_{ij}\phi_i(\param)\phi_j(\param)\Yb_{ij}. \eqnlabel{simpleYb}
%\end{align}

But, $\Yp = \Yp^T$
\begin{align}
\implies \Yb_{ij} = \Yb_{ij}^T. \eqnlabel{Yij}
\end{align}

Combining \eqn{Ybar} and \eqn{Yij} we observe that $\Yb$ admits a linear space of dimension $n(n+1)(N+1)(N+2)/4$.
\end{proof}

Matrix variable $\W(\param)$ is parameterized to be linearly dependent on polynomial chaos basis functions $\phi_i(\param)$, i.e.
\begin{align}
\Wp = \sum_{i=0}^N \W_i\phi_i(\param).
\end{align}

From \eqn{simpleYb}, we can write
\begin{align}
\Yp &= \sum_{ij}\phi_i(\param)\phi_j(\param)\Yb_{ij}, \eqnlabel{Ypsum} \\
&= \left(\mo{\psi}^T(\param) \otimes \I{n} \right)\VYb = \mo{\psi}^T_n(\param)\VYb,
\end{align}
where $\mo{\psi}^T_n(\param) := \mo{\psi}^T(\param) \otimes \I{n}$,
$$
\mo{\psi}(\param):= \begin{bmatrix} \phi_0^2(\param) \\ 2\phi_1(\param)\phi_0(\param)\\ \vdots \\ 2\phi_N(\param)\phi_{N-1}(\param) \\ \phi_N^2(\param) \end{bmatrix}
$$
and
\begin{align}
\VYb:= \begin{pmatrix} \Yb_{00} \\ \Yb_{10} \\ \vdots \\ \Yb_{N(N-1)} \\ \Yb_{NN}\end{pmatrix}\eqnlabel{VYdef}
\end{align}

Matrix variable $\Wp$ can be written in similar form,
\begin{align}
 \Wp &= \sum_{i=0}^N \W_i \phi_i(\param) = \begin{bmatrix} \phi_0(\param)\I{m} & \cdots & \phi_N(\param)\I{m}\end{bmatrix}\underbrace{\begin{pmatrix}\W_0 \\\vdots \\\W_N \end{pmatrix}}_{\VW}  =  \Phin{m}^T(\param)\VW. \eqnlabel{VWdef}
 \end{align}

From the definition of $\mY$
\begin{align}
\mY &:= \I{N+1}\otimes\Yp = \I{N+1} \otimes\Big(\mo{\psi}^T_n(\param) \VYb\Big) = \Big(\I{N+1} \otimes\ \mo{\psi}^T_n(\param)\Big)(\I{N+1}\otimes\VYb),\nonumber \\
&= \Big(\I{N+1} \otimes\ \mo{\psi}^T_n(\param)\Big)\mVYb.
 \eqnlabel{mY2}
\end{align}

Similarly,
\begin{align} \mW &= \I{N+1}\otimes \left(\Phin{m}^T(\param)\VW\right) = \Big(\I{N+1}\otimes \Phin{m}^T(\param)\Big)\mVW.
\eqnlabel{mW2}\end{align} 

We next present the synthesis algorithm for the particular parameterization considered here.
\begin{theorem}\label{thm:Galerkin}
Controller gain
$$\Kp = \left(\sum_{i=0}^N \W_i\phi_i(\param)\right)\left(\Phin{n}^T(\param)\Yb\Phin{n}(\param)\right)^{-1},$$ minimizes \eqn{costFcn} if matrices $\Yb = \Yb^T > 0 \in \real^{n(N+1)\times n(N+1)}$ and $\W_i\in\real^{m\times n}$,  are the solution of the following optimization problem

%\begin{align}
%& \max_{\Yb,\W_i} \tr \left(\sum_{i=0}^N \E{\phi_i^2(\param)}\Yb_{ii}\right) \notag\\[3mm]
%\text{subject to } &\begin{bmatrix}
%\mVYb^T \M_1 + \mVW^T \M_2 + (\cdot)^T & \mVYb^T\sqrt{\M_3} & \mVW^T\sqrt{\M_4}\\[1mm]
%\sqrt{\M_3}\mVYb  & -\I{(N+1)^2(N+2)n/2} &  \mo{0}_{(N+1)^2(N+2)n/2\times m(N+1)^2} \\[1mm]
%\sqrt{\M_4}\mVW  &\mo{0}_{m(N+1)^2 \times (N+1)^2(N+2)n/2} & -\I{m(N+1)^2}
%\end{bmatrix} \leq 0,
%\eqnlabel{lmi1}
%\end{align}

\begin{align}
& \max_{\Yb,\W_i} \tr \left(\sum_{i=0}^N \E{\phi_i^2(\param)}\Yb_{ii}\right) \notag\\[3mm]
&\text{subject to } \\ &\begin{bmatrix}
\sym{\mVYb^T \M_1 + \mVW^T \M_2} & \mVYb^T\sqrt{\M_3} & \mVW^T\sqrt{\M_4}\\[1mm]
\sqrt{\M_3}\mVYb  & -\I{} &  \mo{0} \\[1mm]
\sqrt{\M_4}\mVW  &\mo{0} & -\I{}
\end{bmatrix} \leq 0,
\eqnlabel{lmi1}
\end{align}

%[((mvY'*M1 + mvW'*M2) + (mvY'*M1 + mvW'*M2)') Z1' Z2'; 
%    Z1 -eye(nBasis^2*(nBasis+1)*n/2) zeros(nBasis^2*(nBasis+1)*n/2,m*nBasis^2);
%    Z2 zeros(m*nBasis^2,nBasis^2*(nBasis+1)*n/2) -eye(m*nBasis^2)];
%

where $\mVYb := \I{N+1}\otimes \VYb, \mVW := \I{N+1} \otimes \VW$, $\VYb, \VW$ are functions of $\Yb$ and $\W_i$ defined in \eqn{VYdef} and \eqn{VWdef} respectively,
\begin{align}
\M_1 & := \E{(\I{N+1}\otimes \mo{\psi}_n)\Phin{n}\A^T\Phin{n}^T},\label{eq:M1}\\
\M_2 & := \E{(\I{N+1}\otimes \Phin{m})\Phin{m}\B^T\Phin{n}^T},\\
\M_3 & := \E{(\I{N+1}\otimes \mo{\psi}_n)\Phin{n}\Q\Phin{n}^T(\I{N+1}\otimes \mo{\psi}^T_n)}, \\
\M_4 & := \E{(\I{N+1}\otimes \Phin{m}) \Phin{m}\R\Phin{m}^T (\I{N+1}\otimes \Phin{m}^T)}\label{eq:M4},
\end{align}
and $\sqrt{\M_3},\sqrt{\M_4}$ are the principal square roots of the respective matrices.
\label{thm:lmi_galerkin}
\end{theorem}

\begin{proof}
Recall from \eqn{Ypsum}, $\Yp = \sum_{ij}\phi_i(\param)\phi_j(\param)\Yb_{ij}$. Noting that $\E{\phi_i(\param)\phi_j(\param)} = 0$, for $i\neq j$, the cost function in \eqn{cost} is then
\begin{align}
\tr \E{\Yp} = \tr \Big(\sum_{i=0}^N \E{\phi_i^2(\param)}\Yb_{ii}\Big).
\end{align}

From \eqn{mY2} and \eqn{mW2}, we can substitute $\mY$ and $\mW$ in \eqn{constr} to get
\begin{align}
&\mathsf{sym}\left(\mVYb^T \E{(\I{N+1}\otimes \mo{\psi}_n)\Phin{n}\A^T\Phin{n}^T} +  \mVW^T\E{(\I{N+1}\otimes \Phin{m})\Phin{m}\B^T\Phin{n}^T}\right) + \nonumber \\
&\mVYb^T \E{(\I{N+1}\otimes \mo{\psi}_n)\Phin{n}\Q\Phin{n}^T(\I{N+1}\otimes \mo{\psi}_n^T)}\mVYb + 
\mVW^T \E{(\I{N+1}\otimes \Phin{m})\Phin{m}\R\Phin{m}^T(\I{N+1}\otimes \Phin{m}^T)}\mVW \leq 0. 
\end{align}
Applying Schur complement we get the LMI in \eqn{lmi1}.
\end{proof}

\subsection{Stochastic Collocation Based Formulation}
In this section, we solve the synthesis problem derived in theorem \ref{thmMain} in the stochastic collocation framework. In this framework, we can parameterize the matrix variables in \eqn{quad2} as
\begin{align}
&\vo{z}(\param) = \Ln{n}(\param)^T\zsc,\label{eq:z_sc}\\
&\Y(\param)=\Ln{n}^T(\param)\Yt\Ln{n}(\param),\label{eq:Y_sc}\\
&\W(\param)=\Ln{m}^T(\param)\vo{V_{\tilde{\W}}}\label{eq:W_sc},
\end{align}
where $\Ln{n}=\begin{bmatrix}
L_0(\param)\\\vdots\\L_N(\param)
\end{bmatrix}
\otimes \I{n}$, $\Yt=
\begin{bmatrix}
\Yt_{00} & \cdots & \Yt_{0N}\\
\vdots & \ddots & \vdots\\
\Yt_{N0} & \cdots & \Yt_{NN}
\end{bmatrix}
$, and $
\vo{V}_{\tilde{\W}}=
\begin{bmatrix}
\tilde{\W}_0\\
\vdots\\
\tilde{\W}_N
\end{bmatrix}
$. Based on this parameterization, we have the following optimization problem for synthesis.

\begin{theorem} \label{thm:SC}
Controller gain
$$\Kp = \left(\Ln{m}^T(\param)\vo{V_{\tilde{\W}}}\right)\left(\Ln{n}^T(\param)\Yt\Ln{n}(\param)\right)^{-1},$$ minimizes \eqn{costFcn} if $\exists$ matrices $\Yt = \Yt^T > 0 \in \real^{n(N+1)\times n(N+1)}$ and $\tilde{\W}_i\in\real^{m\times n}$,  that solves the following optimization problem

\begin{align}
& \max_{\Yb,\W_i} \tr \left(\sum_{i=0}^N \E{\Ln{in}}\Yt_{ii}\right) \notag\\[3mm]
&\text{subject to } \\
&\begin{bmatrix}
\vo{M}_{11,i} & \vo{M}_{12,i} & \vo{M}_{13,i}\\[1mm]
\vo{M}^T_{12,i}  & -\I{} &  \mo{0} \\[1mm]
\vo{M}^T_{13,i}  &\mo{0} & -\I{}
\end{bmatrix} \leq 0 \quad \text{for }i=0,1,\cdots,N;
\eqnlabel{lmi_sc}
\end{align}

where 
\begin{align}
\vo{M}_{11,i} &:= \sym{\Yt_{ii}^T \E{\Ln{in}}\A^T(\param_i) + \tilde{\W}_i^T \B^T(\param_i)},\\
\vo{M}_{12,i} &:= \Yt_{ii}^T\sqrt{\E{\Ln{in}}\Q},\\
\vo{M}_{12,i} &:=  \tilde{\W}_i^T\sqrt{\E{\Ln{in}}\R},
\end{align}
$\Ln{in}=L_i\otimes\I{n}$, and $\param_i$ are the roots of the polynomial chaos basis of degree $N+1$.
\label{thm:lmi_sc}
\end{theorem}

\begin{proof}
Substituting \eqref{eq:Y_sc} into \eqn{cost} and applying Lemma \ref{lem:gaussian}, we have
\begin{align*}
&\tr \E{\Yp}=\tr \E{\Ln{n}^T(\param)\Yt\Ln{n}(\param)}\\
&=\tr(\E{\sum_{i=0}^N\sum_{j=0}^N\Ln{in}\Yt_{ij}\Ln{jn}})\\
&\approx \tr(
\E{\Ln{0n}}\Yt_{00}+\E{\Ln{1n}}\Yt_{11}+\cdots+\E{\Ln{Nn}}\Yt_{NN}
)\\
&=\tr(\sum_{i=0}^{N} \E{\Ln{in}}\Yt_{ii}),
\end{align*}
which is the cost function we have to maximize. Then, substituting \eqref{eq:z_sc}-\eqref{eq:W_sc} into \eqn{quad2} yields

\begin{align*}
 \boldsymbol{\mathsf{E}}\left[ \zsc^T\Ln{n} \left( \sym{\Ln{n}^T\Yt\Ln{n}\A^T  + \vo{V}_{\tilde{\W}}^T\Ln{m}\B^T}\right.+ \left. \Ln{n}^T\Yt\Ln{n}\Q\Ln{n}^T\Yt\Ln{n} + \vo{V}_{\tilde{\W}}^T\R\Ln{m}^T\vo{V}_{\tilde{\W}}\right)\Ln{n}^T\zsc\right]  \leq 0
\end{align*}
or
\begin{align}
 \boldsymbol{\mathsf{E}}\left[\Ln{n} \left( \sym{\Ln{n}^T\Yt\Ln{n}\A^T + \vo{V}_{\tilde{\W}}^T\Ln{m}\B^T}  \right.+\left.\Ln{n}^T\Yt\Ln{n}\Q\Ln{n}^T\Yt\Ln{n} + \vo{V}_{\tilde{\W}}^T\R\Ln{m}^T\vo{V}_{\tilde{\W}}\right) \Ln{n}^T\right] \leq 0.\eqnlabel{quadsc}
\end{align}

Applying the Lemma \ref{lem:gaussian}, \eqn{quadsc} can be represented as
\begin{align}
\E{
\begin{matrix}
\Ln{0n}&&\\
&\ddots&\\
&&\Ln{Nn}
\end{matrix}
}
\begin{bmatrix}
\vo{G}_0&&\\
&\ddots&\\
&&\vo{G}_N
\end{bmatrix} \leq 0,
\eqnlabel{lmi_sc1}
\end{align}
where
\begin{align*}
\vo{G}_i &\approx
\sym{\Ln{n,i}^T\Yt\Ln{n,i}\A^T(\param_i) + \vo{V}_{\tilde{\W}}^T\Ln{m,i}\B^T(\param_i)} + \Ln{n,i}^T\Yt\Ln{n,i}\Q\Ln{n,i}^T\Yt\Ln{n,i} + \vo{V}_{\tilde{\W}}^T\R\Ln{m,i}^T\vo{V}_{\tilde{\W}},\\
&=\sym{\Yt_{ii}\A^T(\param_i) + \tilde{\W}_i^T\B^T(\param_i)} + \Yt_{ii}\Q\Yt_{ii} + \tilde{\W}_i^T\R\tilde{\W}_i.
\end{align*}

We use the notations $\Ln{n,i}=\Ln{n}(\param_i)$ and $\Ln{m,i}=\Ln{m}(\param_i)$ to simplify the above expressions. Since \eqn{lmi_sc1} is in a diagonal form, it can be separated into $N+1$ independent constraints.

\begin{align}
\E{\Ln{in}}\left[\sym{\Yt_{ii}\A^T(\param_i) + \tilde{\W}_i^T\B^T(\param_i)}+ \Yt_{ii}\Q\Yt_{ii} + \tilde{\W}_i^T\R\tilde{\W}_i\right] \leq 0 \quad \text{for }i=0,\cdots,N; 
\end{align}
or 
\begin{align}
\sym{\Yt_{ii}\E{\Ln{in}}\A^T(\param_i)  + \tilde{\W}_i^T\E{\Ln{in}}\B^T(\param_i)} + \Yt_{ii}\E{\Ln{in}}\Q\Yt_{ii} + \tilde{\W}_i^T\E{\Ln{in}}\R\tilde{\W}_i \leq 0,
\eqnlabel{quad_sc}
\end{align}
for $i=0,\cdots, N$. Applying Schur complement to \eqn{quad_sc} we obtain $N+1$ final LMIs as \eqn{lmi_sc}.

\end{proof}

\subsection{Stability Concern Due to Finite Term Polynomial Chaos Expansion}
Theorem \ref{thmMain} presents the optimization problem for synthesis assuming infinite term polynomial chaos expansion of $\x(t,\param)$. There are no approximations in that problem formulation. However, we solve this problem using finite terms in the expansion, for both Galerkin projection and stochastic collocation framework. The problem formulations in theorems \ref{thm:Galerkin} and \ref{thm:SC} are based on finite term expansion of $\x(t,\param) \approx \hat{\x}(t,\param) := \sum_{i=0}^N \x_i(t)\phi_i(\param)$. Therefore, optimal control of $\hat{\x}(t,\param)$ does not necessarily imply optimal control of $\x(t,\param)$. In fact, we cannot conclude $\lim_{t\to\infty}\Exp{\|\hat{\x}(t,\param)\|^2_2} \to 0 \implies \lim_{t\to\infty}\Exp{\|\x(t,\param)\|_2^2} \to 0$. That is, we can cannot guarantee exponential mean square stability (EMS) of $\x(t,\Delta)$ from the EMS stability of $\hat{\x}(t,\Delta)$. 

To circumvent this problem, we guarantee stability of $\x(t,\Delta)$ in the \textit{worst-case} sense by imposing the following additional constraints,
\begin{align}\sym{\A(\param_\text{wc})\Y(\param_\text{wc})  + \B(\param_\text{wc})\W(\param_\text{wc})} < 0,
\end{align}
where $\param_\text{wc}$ represents the worst-case values from $\domain{\param}$.

Therefore, the results presented in this paper can be interpreted as synthesis of parameter dependent gain $\K(\param)$ that stabilizes the system in \eqn{lpv} in the worst-case sense and optimizes the performance, using theorems \ref{thm:Galerkin} and \ref{thm:SC}, based on the first $N$ modes of $\x(t,\Delta)$.

\section{Example}
We next consider an autopilot design for a nonlinear missile model \cite{wu1995lpv} using the results presented in this paper and benchmark it with existing techniques. The dynamics of the missile model is given by

\begin{align}
&\dot{\alpha}=K_{\alpha}MC_n(\alpha,\delta)\cos(\alpha)+q,\\
&\dot{q}=K_{q}M^2C_m(\alpha,\delta),
\end{align}
where
\begin{align*}
&C_n(\alpha,\delta,M)=\alpha\left[a_n|\alpha|^2+b_n|\alpha|+c_n\left(2-\frac{M}{3}\right)\right]+d_n\delta,\\
&C_m(\alpha,\delta,M)=\alpha\left[a_m|\alpha|^2+b_m|\alpha|+ c_m\left(-7+\frac{8M}{3}\right)\right]+ d_m\delta,
\end{align*}
are the aerodynamic coefficients, $\alpha$ is angle of attack in degrees, $q$ is pitch rate in degrees per second, $\delta$ is tail deflection  angle in degrees, and $M=2.5$ is Mach number. For simplicity, we consider $M$ to be constant. The system parameters are defined in \cite{wu1995lpv}. We transform the nonlinear dynamics to a quasi-LPV system by introducing $\rho:=\alpha$,
\begin{align}
\begin{bmatrix}
\dot{\alpha}\\
\dot{q}
\end{bmatrix}
= 
\begin{bmatrix}
K_{\alpha}M\left[a_n|\rho|^2+b_n|\rho|+c_n\left(2-\frac{M}{3}\right)\right]\cos(\rho) & 1\\
K_qM^2\left[a_m|\rho|^2+b_m|\rho|+c_m\left(-7+\frac{8M}{3}\right)\right] & 0
\end{bmatrix}
\begin{bmatrix}
\alpha\\q
\end{bmatrix} + 
\begin{bmatrix}
K_{\alpha}Md_n\cos{\rho}\\
K_qM^2d_m
\end{bmatrix}
\delta.
\label{eq:missile_lpv}
\end{align}

The objective is to design a full state feedback controller $\K(\rho)$ that stabilizes the missile system such that $-20^{\circ}\leq\alpha\leq 20^{\circ}$ while minimizing the cost-to-go function with
\begin{align*}
\Q=\begin{bmatrix}
0.2&0\\0&0.2
\end{bmatrix},
\quad \R=1.
\end{align*}
We design four different control systems: $\K_{\text{LTI}}$, $\K_{\text{LPV}}$,$\K_{\text{pcLPV}}$,$\K_{\text{scLPV}}$, which are synthesized as follows.
\begin{itemize}
\item $\K_{\text{LTI}}$
\item[] We linearize the nonlinear dynamics about $(0, 0)$ to get $(\A_{\text{LTI}},\B_{\text{LTI}})$ and the controller $\K_{\text{LTI}}=
\W_{\text{LTI}}\Y_{\text{LTI}}^{-1}$ is obtained by solving the optimization problem:
\begin{align*}
&\max_{\Y_\text{LTI},\W_\text{LTI}}  \tr \Y_\text{LTI}, \text{ subject to} \\
&\begin{bmatrix}
\sym{\Y_\text{LTI}\A_\text{LTI}^T +  \W_\text{LTI}^T\B_\text{LTI}^T}  & \Y_\text{LTI} & \W^T_\text{LTI} \\
\Y_\text{LTI} & -\Q^{-1} & \mo{0} \\
\W_\text{LTI} & \mo{0} & -\R^{-1}
\end{bmatrix} \leq 0.
\end{align*}
\item $\K_{\text{LPV}}$
\item[] According to the LPV system \eqref{eq:missile_lpv}, we  set $ \Y_\text{LPV}(\rho) :=\Y_0 + \rho \Y_1 > 0, \Y_i = \Y_i^T$, and $\W_\text{LPV}(\rho) :=\W_0 + \rho \W_1$, and the controller $\K_{\text{LPV}}=\W_{\text{LPV}}\Y_{\text{LPV}}^{-1}$ is obtained by solving the optimization problem below with 2, 20, 50, 100 sample points.
\begin{align*}
&\max_{\Y_0,\Y_1,\W_0,\W_1}  \tr (\Y_0 + \Y_1), \text{ subject to} \\
&\begin{bmatrix}
\begin{matrix}\mathsf{sym}\left(\Y(\rho_k)\A^T(\rho_k)\right.\\  \left.\W^T(\rho_k)\B^T(\rho_k)\right) \end{matrix} & \Y(\rho_k) & \W^T(\rho_k) \\
\Y(\rho_k) & -\Q^{-1} & \mo{0} \\
\W(\rho_k) & \mo{0} & -\R^{-1}
\end{bmatrix} \leq 0, 
\end{align*}
for $\rho_k \in [-20, 20]$.
The classical LPV synthesis algorithm is recovered when $\rho_k = \{-20, 20\}$, i.e. $\rho_k$ takes the  extreme values.
\item $\K_{\text{pcLPV}}$
\item[] We assume the parameter $\rho$ is a random variable uniformly distributed over $[-20, 20]$, so we define $\rho\equiv\param\in \uniform{[-20,20]}$ and substitute it into \eqref{eq:missile_lpv}. From theorem \ref{thm:lmi_galerkin}, the controller $\K_{\text{pcLPV}}$ is obtained with 3rd, 4th, and 5th order polynomial chaos expansion.
\item $\K_{\text{scLPV}}$ from theorem \ref{thm:lmi_sc}, the controller $\K_{\text{scLPV}}$ is obtained with 5th, 9th, and 12th polynomial chaos expansion.
\end{itemize}

The controllers were synthesized in \texttt{MATLAB} \cite{matlab} using \texttt{CVX} \cite{grant2008cvx}.

\Fig{all} shows the comparison of performance between different control systems, and table \ref{table1} compares the synthesis time and cost-to-go, over $t\in[0, 20]$ seconds, for each controller. From \fig{all} and table \ref{table1}, we observe that $\K_{\text{scLPV}}$ achieves the best tradeoff between synthesis time and closed-loop performance.

\Fig{LPV}, shows the performance of $\K_{\text{LPV}}$ for various number of samples included in the synthesis. For two samples, the trajectories do not converge to zero, hence the cost-to-go is infinite. In table \ref{table1}, the cost for $\K_{\text{LPV}}$, over $t\in[0,20]$ seconds, is worse than $\K_{\text{LTI}}$, but improves as number of samples are increased. The computational time also increases  with sample size. \Fig{pc}, shows the performance of $\K_{\text{pcLPV}}$ for various orders of polynomial chaos expansion and provides the best performing controller, at the expense of very large synthesis times.  \Fig{sc}, shows the performance of $\K_{\text{scLPV}}$ for various orders of stochastic collocation and provides performance comparable with $\K_{\text{pcLPV}}$, but with much less synthesis time. In summary, the stochastic collocation approach provides a computational efficient framework for synthesizing LPV controllers, and performs better than the classical LPV designs.

\section{Summary}
In this paper, we presented a new framework for synthesizing LPV controllers using polynomial chaos framework. This framework builds on the probabilistic representation of the scheduling variables. We treat the LPV system as a stochastic linear system, where the parameter is treated as a random variable with a given distribution. This paper has taken this approach to develop  new algorithms for synthesis of linear quadratic regulators. We pursued two approaches: Galerkin projection and stochastic collocation to develop the necessary theoretical framework. The synthesis algorithms were applied to a regulator design problem for a nonlinear missile system, which significantly outperformed controllers synthesized using classical LPV design techniques. We also presented tradeoff between synthesis time and closed-loop performance for the various methods considered. Based on our study, we concluded that
the stochastic collocation approach provides a computational efficient framework for synthesizing LPV controllers, and performs better than the classical LPV designs.

\begin{table}
\begin{center}
\begin{tabular}[h!]{l|ccl}
\hline Controller & Synthesis & \# of SDP  & Cost-to-go \\
                  & Time (sec) & Variables &   \\ \hline
$\K_\text{LTI}$ & 0.8121 & 7 & 209.3011 \\
$\K_\text{LPV}$ (2 samples) & 1.0635 & 14 & 220.3811\footnotemark \\
$\K_\text{LPV}$ (20 samples) & 1.4584 & 140 & 193.2252   \\
$\K_\text{LPV}$ (50 samples) & 3.1821 & 350 & 175.0896  \\
$\K_\text{LPV}$ (100 samples) & 6.5485 & 700 & 169.3382\\
$\K_\text{pcLPV}$ ($3^\text{rd}$ order PC) & 124.6578  & 216 &  104.0365\\
$\K_\text{pcLPV}$ ($4^\text{th}$ order PC) & 594.4620  & 380 &  103.9989\\
$\K_\text{pcLPV}$ ($5^\text{th}$ order PC) & 1189.0534 & 612 &  103.9018 \\
$\K_\text{scLPV}$ ($5^\text{th}$ order SC) & 2.0419 & 28 &  104.5838 \\
$\K_\text{scLPV}$ ($9^\text{th}$ order SC) & 2.8364 & 70 &  104.0548 \\
$\K_\text{scLPV}$ ($12^\text{th}$ order SC) &3.3598 & 98 &  103.9075 \\
\hline
\end{tabular}
\end{center}
\caption{Comparison of controller performances and synthesis times.}
\label{table1}
\end{table}
\footnotetext{The cost-to-go is computed over $t\in[0\;20]$. Since the states do not converge to zero for this case, the actual cost-to-go is infinity.}

\begin{figure}
\includegraphics[width=\linewidth]{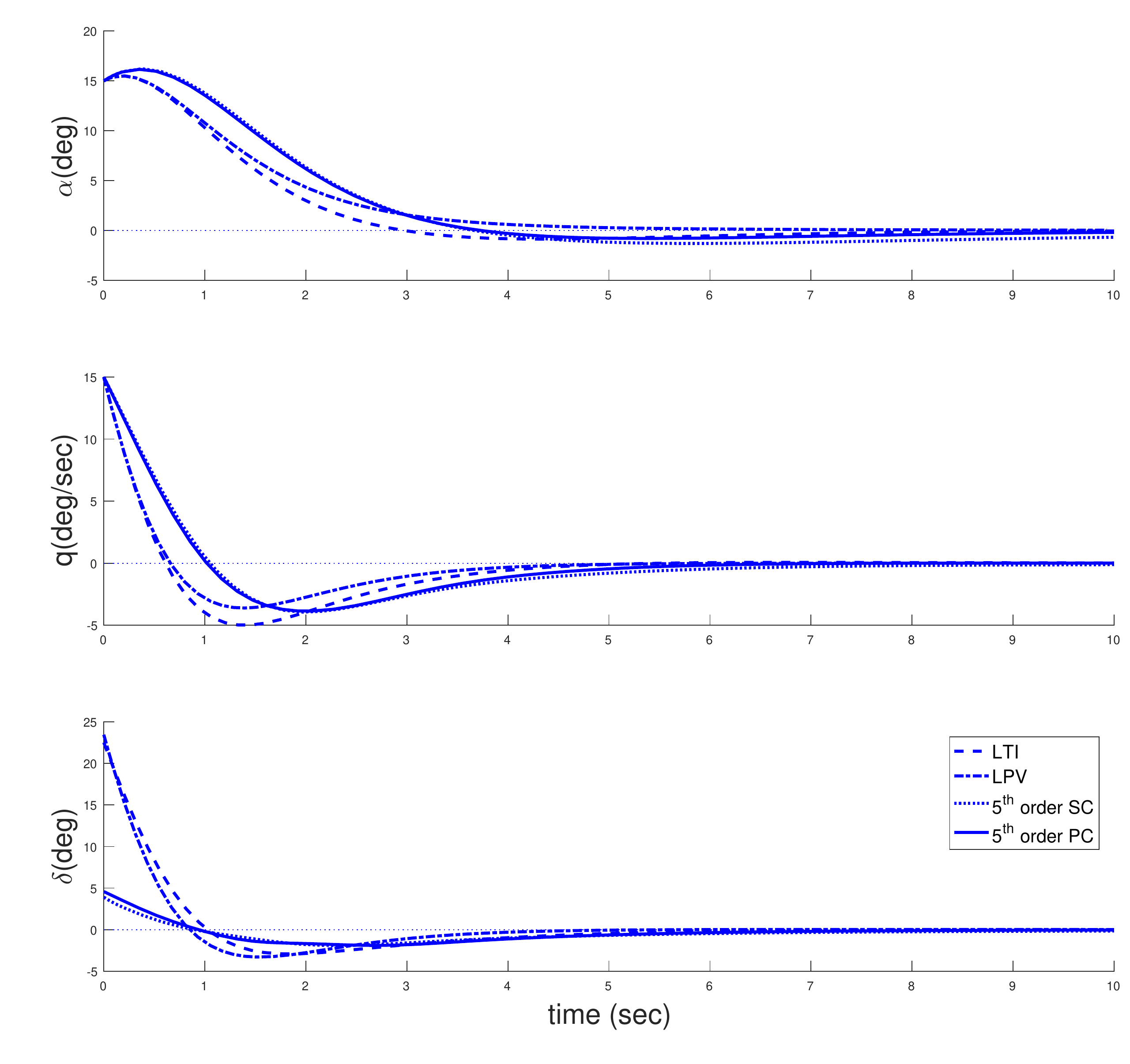}
\caption{State and control trajectories for the missile autopilot by applying different controllers.}
\figlabel{all}
\end{figure} 

\begin{figure}
\includegraphics[width=\linewidth]{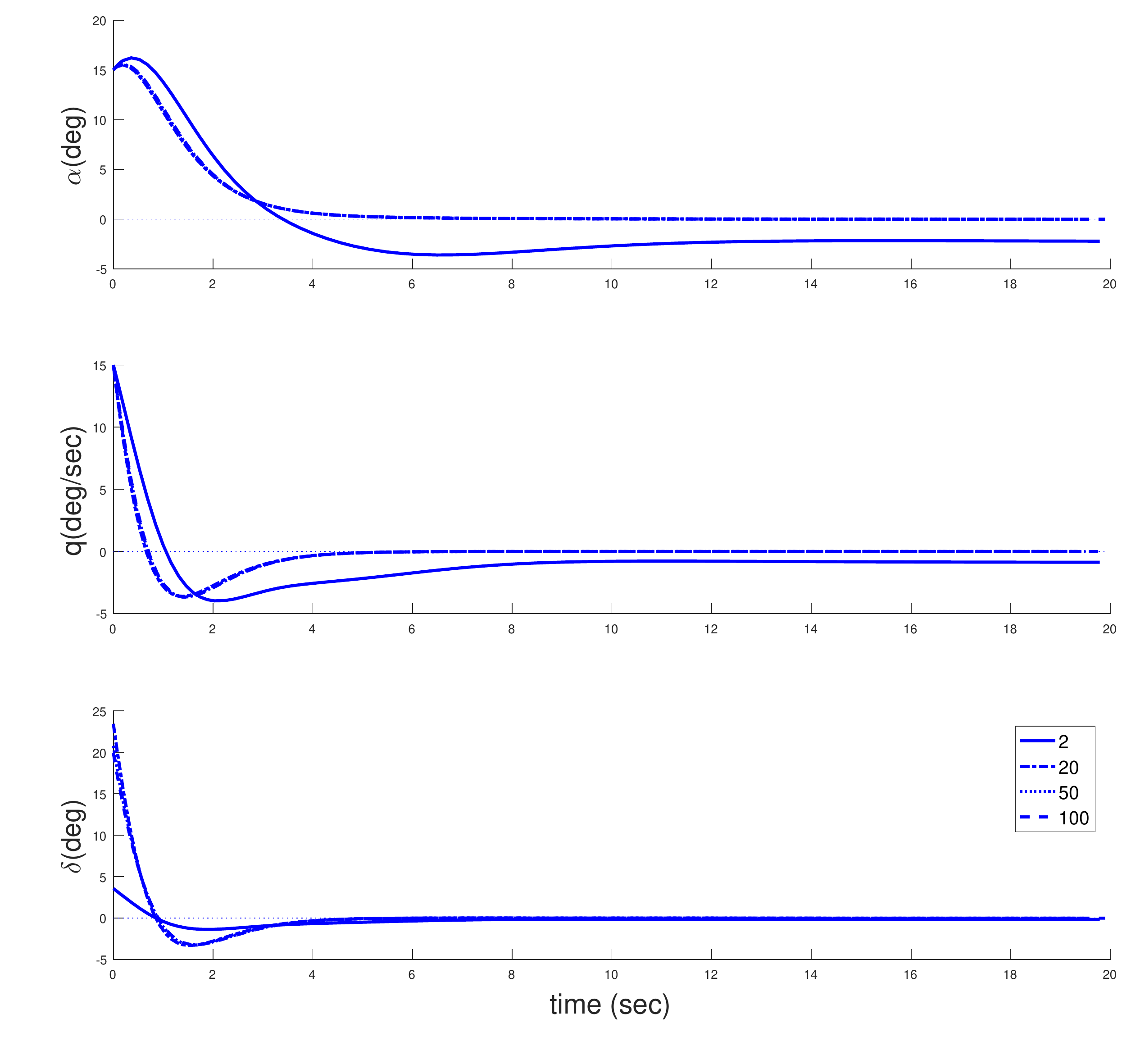}
\caption{State and control trajectories for the missile autopilot by applying LPV control with different sample points.}
\figlabel{LPV}
\end{figure} 

\begin{figure}
\includegraphics[width=\linewidth]{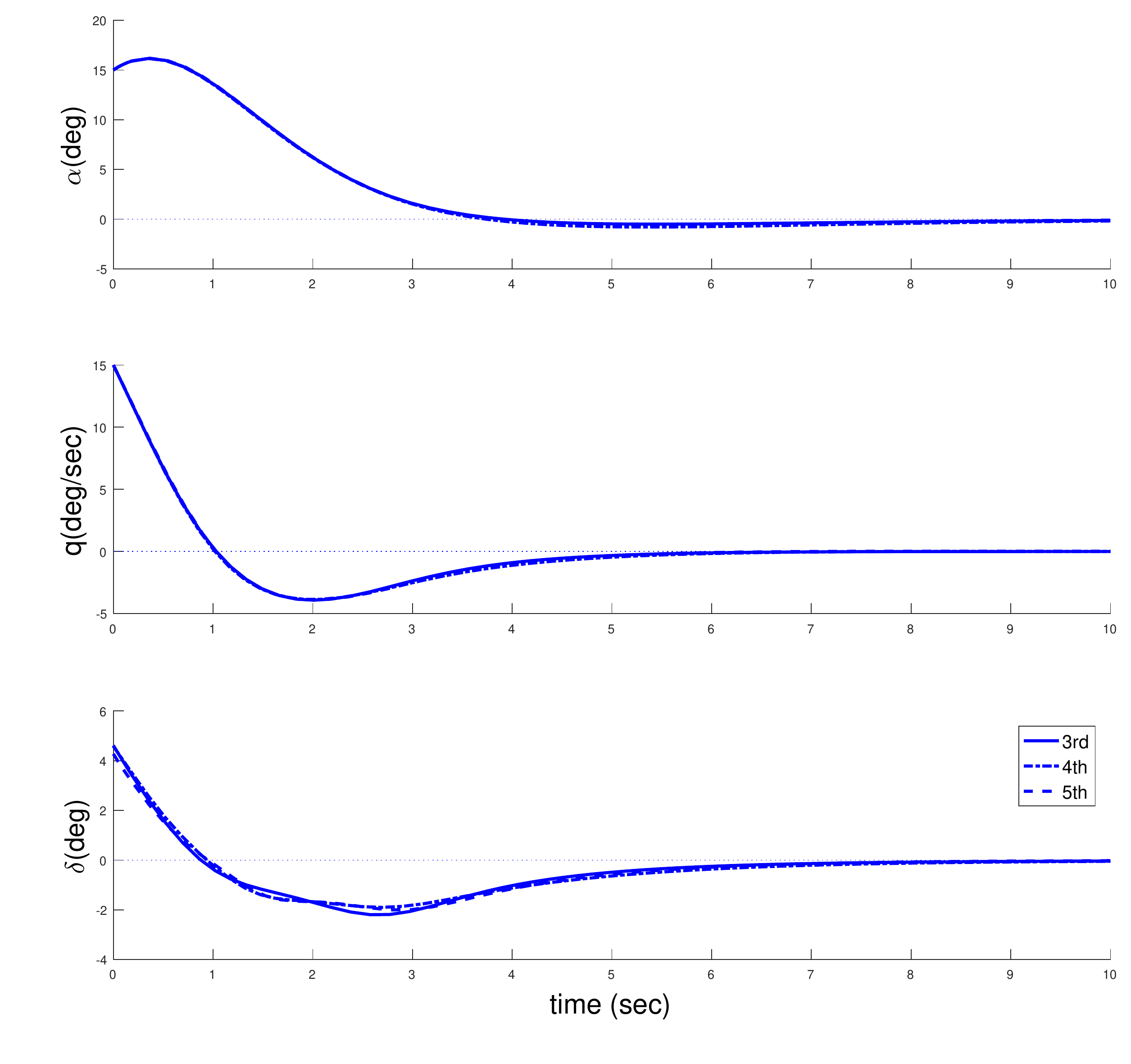}
\caption{State and control trajectories for the missile autopilot by applying pcLPV control with different order polynomial chaos expansion.}
\figlabel{pc}
\end{figure} 

\begin{figure}
\includegraphics[width=\linewidth]{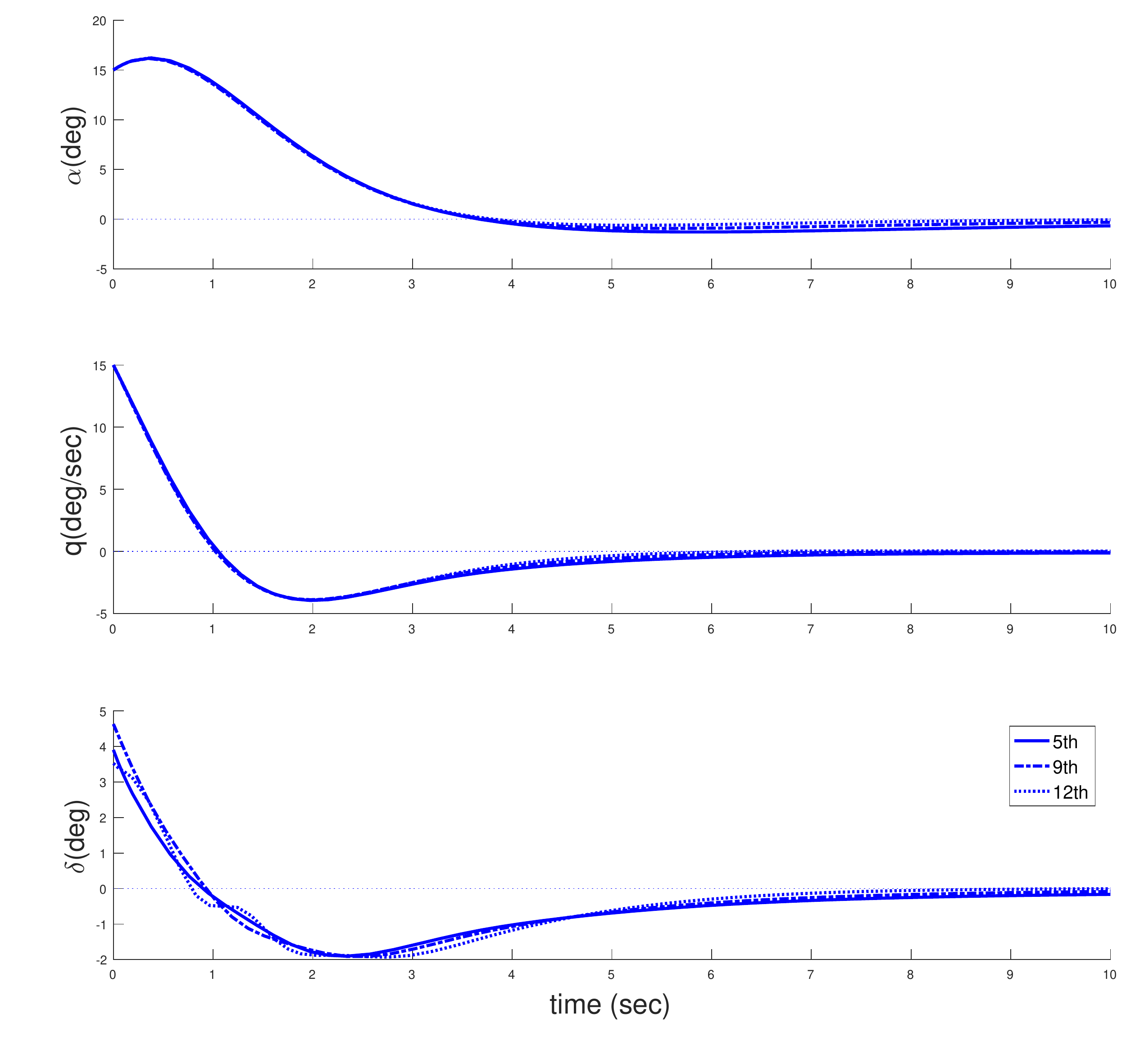}
\caption{State and control trajectories for the missile autopilot by applying scLPV control with different sample points.}
\figlabel{sc}
\end{figure} 

\section*{Acknowledgements}
This work was supported by the TIAS Award of Heep Fellowship.

\bibliographystyle{unsrt}
\bibliography{raktim}

%\begin{IEEEbiography}[{\includegraphics[width=1in,height=1.25in,clip,keepaspectratio]{photo_HSU}}]{Shao-Chen Hsu}
%received his Bachelors in Mechanical Engineering from National Central University, Taoyuan, Taiwan, in 2009, and his Masters in Mechanical Engineering from Texas A\&M University in 2014. He is currently a Ph.D. candidate in Aerospace Engineering at Texas A\&M University. His major research interests include
%stochastic controls, polynomial chaos theory, tensegrity structure, and robotics.
%\end{IEEEbiography}
%\enlargethispage{-10cm}
%\begin{IEEEbiography}[{\includegraphics[width=1in,height=1.25in,clip,keepaspectratio]{photo_bhattacharya}}]{Raktim Bhattacharya}
%received his M.S. and Ph.D. degrees in aerospace engineering from the University of Minnesota in 2000 and 2003, respectively. He joined the aerospace engineering department of Texas A\&M University in 2005 and currently is an associate professor. His research interests include robust control, nonlinear systems, uncertainty quantification and convex optimization. 
%\end{IEEEbiography}
%

\end{document}